\documentclass[review]{elsarticle}

\usepackage{hyperref}
\usepackage{amssymb,amsmath,amsfonts,amsthm}
\usepackage{latexsym}
\usepackage{graphicx}
\usepackage{url}

\newcommand{\seqnum}[1]{\href{http://oeis.org/#1}{\underline{#1}}}
\newtheorem{theorem}{Theorem} 
\newtheorem{lemma}{Lemma} 
\newtheorem{proposition}{Proposition}

\newtheorem{corollary}{Corollary}

\journal{Theoretical Computer Science}

\begin{document}

\begin{frontmatter}

\title{The number of valid factorizations of Fibonacci prefixes}

\author[I2M]{Pierre Bonardo}

\author[I2M]{Anna E. Frid}
\ead{anna.e.frid@gmail.com}
\ead[url]{http://iml.univ-mrs.fr/~frid/}
\author[Waterloo]{Jeffrey Shallit}
\ead{shallit@uwaterloo.ca}
\ead[url]{https://cs.uwaterloo.ca/~shallit/}

\address[I2M]{Aix Marseille Univ, CNRS, Centrale Marseille, I2M, Marseille, France}
\address[Waterloo]{School of Computer Science, University of Waterloo, Waterloo, ON  N2L 3G1, Canada}

\begin{abstract}
We establish several recurrence relations and an explicit formula 
for $V(n)$, the number of factorizations of the length-$n$ prefix of the Fibonacci word into a (not necessarily strictly) decreasing sequence of standard Fibonacci words. In particular, we show that
the sequence $V(n)$ is the shuffle of the 
ceilings of two linear functions of $n$. 
\end{abstract}


\end{frontmatter}


\section{Introduction}
In the classical Fibonacci, or Zeckendorf, numeration system \cite{lekkerkerker,zeckendorf}, a positive integer is represented as a sum of Fibonacci numbers:
$$n=F_{m_k}+F_{m_{k-1}}+\cdots + F_{m_0},$$ where
$m_k > m_{k-1} > \cdots > m_0 \geq 2$ and, as usual,
$F_0=0$, $F_{1}=1$, and $F_{m+2}=F_{m+1}+F_{m}$ for all $m\geq 0$.
For example, $16=13+3=F_7+F_4= [100100]_F$, where a digit in 
brackets is $1$ if the respective 
Fibonacci number appears in the sum, and $0$ otherwise. Here a representation ends by the digits corresponding to $F_4=3$, $F_3=2$ and $F_2=1$.
 
Under the condition that $m_{i}$ and $m_{i+1}$ are never consecutive, 
that is, $m_{i+1}-m_i\geq 2$, or,
equivalently, that the Fibonacci numbers $F_i$ 
are chosen greedily, such a {\it canonical} representation is unique, and the
language $L_V$ of all canonical representations is given by the regular
expression $\epsilon + 1(0 + 01)^*$, where the empty word $\epsilon$ is the representation of 0.    
At the same time, if consecutive Fibonacci numbers are allowed, but at most once each, the number of such {\it legal} representations of $n$ is the well-known integer sequence \seqnum{A000119} from the Online Encyclopedia of Integer Sequences (OEIS) \cite{oeis}. Its values oscillate between $1$ (on numbers of the form $F_{i}-1$) and $\sqrt{n+1}$ (on numbers of the form $n = F_i^2-1$) \cite{stock}.  

For example, since
\begin{align*} 
16&=13+3=8+5+3=8+5+2+1=13+2+1\\ 
&= [100100]_F = [11100]_F = [11011]_F = [100011]_F,  
\end{align*} 
the number of legal representations of 16 is 4. Each legal representation of $n$ can be obtained from a canonical one by a series of replacements  
$$\cdots 100 \cdots \longleftrightarrow \cdots 011 \cdots,$$ 
corresponding to the
replacement of a Fibonacci number $F_{m+2}$ by $F_{m+1}+F_m$. 
 
In this paper, we allow even more freedom in Fibonacci representations of $n$, allowing the transformations  
\begin{equation}\label{e:k0l} 
\cdots k0l \cdots \longleftrightarrow \cdots (k-1)\; 1 (l+1) \cdots 
\end{equation} 
for all $k>0$, $l \geq 0$. Note that the introduced transformation corresponds to passing from a sum of the form $kF_{m+1}+lF_{m-1}$ to the sum $(k-1)F_{m+1}+F_m+(l+1)F_{m-1}$, and, in particular, does not change the represented number. 
 
The representations that can be obtained from the canonical one by a series of transformations as in \eqref{e:k0l} are called {\it valid}, and were introduced in \cite{frid} in a more general setting because of their link to the Fibonacci word and factorizations of its prefixes, as explained below. Clearly, each legal representation is valid, but the opposite is not true. For example, starting from the legal representation $16= [11011]_F$, we can find two more valid representations 
\[16= [10121]_F = [1221]_F ,\] 
and starting from the legal representation $16= [11100]_F$, we find a new representation 
\[16= [20000]_F,\] 
so that the total number of valid representations of 16 is 7. 
 
Let $V(n)$ denote the number of valid representations of $n$. The goal of this paper is to prove a precise formula for $V(n)$, given below in Theorem \ref{t:main}.   Our formula demonstrates that the values of $V(n)$ are determined by the shuffle of two straight lines of irrational slope; see Fig.~\ref{f:1}. 
\begin{figure} 
\centering \includegraphics[width=0.99\textwidth, height=0.2\textheight]{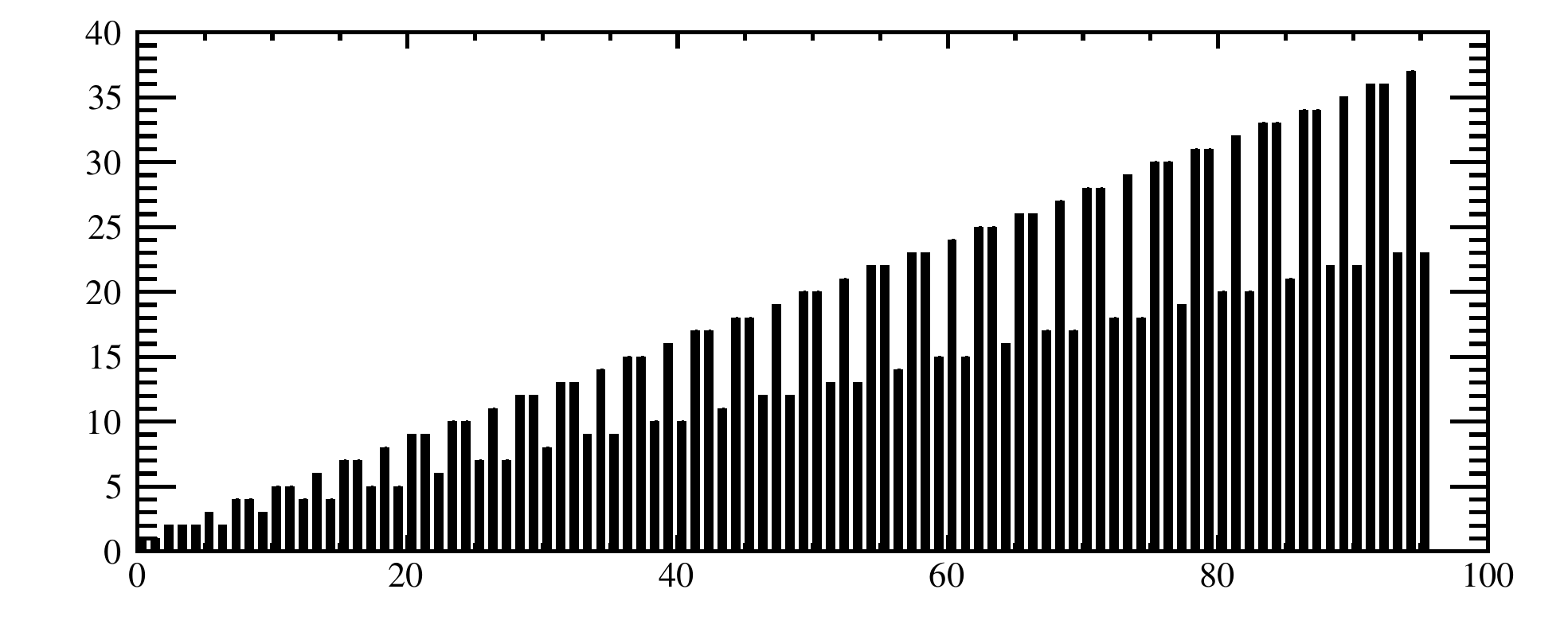} 
\caption{First 100 values of $V(n)$}\label{f:1} 
\end{figure} 
\section{Notation and Sturmian representations} 
\label{notation}
We use notation common in combinatorics on words; the reader is referred, for example, to \cite{lothaire2} for an introduction. Given a finite word $u$, we denote its length by $|u|$. The power $u^k$ just means the concatenation
$u^k=\underbrace{u\cdots u}_{k}$. The $i$'th symbol of a finite or infinite word
$u$ is denoted by $u[i]$, so that $u=u[1]u[2]\cdots$. A factor $w[i+1]w[i+2]\cdots w[j]$ of a finite or infinite word $w$, or, more precisely, its occurrence starting from position $i+1$ of $w$,  is denoted by $w(i..j]$.
In particular, for $j\geq 0$, the word
$w(0..j]$ is the prefix of $w$ of length $j$. 
 
The {\it standard Fibonacci sequence} $(f_n)$ of words over
the binary alphabet $\{a,b\}$ is defined as follows:  
\begin{equation}\label{e:def} 
f_{-1}=b, \quad\quad f_0=a, \quad\quad f_{n+1}=f_nf_{n-1} \mbox{~for all~} n\geq 0. 
\end{equation} 
The word $f_n$ is called also the {\it standard word of order $n$}. In particular, $f_1=ab$, $f_2=aba$, $f_3=abaab$, $f_4=abaababa$, and so on. 
From the definition, we easily see that
the length of $f_n$ is the Fibonacci number $F_{n+2}$.
 
The infinite word $${\bf f} =\lim_{n\to \infty} f_n=abaababaabaababaababa\cdots$$ is called the Fibonacci infinite word.   Here we index it starting
with ${\bf f}[1] = a$.
 
In the {\it Fibonacci}, or {\it Zeckendorf numeration system}, a non-negative integer $N<F_{n+3}$ is represented as a sum of Fibonacci numbers 
\begin{equation}\label{e:numr} 
N=\sum_{0\leq i \leq n} k_i F_{i+2}, 
\end{equation} 
where $k_i \in \{0,1\}$ for $i \geq 0$. In the canonical version of the definition, the following condition holds: 
\begin{equation}\label{e:can_cond} 
 \mbox{for } i \geq 1, \mbox{ if } k_i=1, \mbox{ then } k_{i-1}=0. 
\end{equation} 
Under this nonadjacency condition, the representation of $N$ is unique up to leading zeros. However, by removing the nonadjacency
condition, we can get multiple representations: for example, $14=F_7+F_2=F_6+F_5+F_2=F_6+F_4+F_3+F_2$. We call such representations {\it legal} and denote a representation $N=\sum_{0\leq i \leq n} k_i F_{i+2}$ by $N= [k_n\cdots k_0]_F $. If the condition \eqref{e:can_cond} holds, we call the representation {\it canonical}. 
 
Let $L(n)$ denote the number of legal representations of $n$. The sequence $(L(n))$ is well-studied (see, e.g., \cite{berstel_fib}) and listed in the OEIS as sequence \seqnum{A000119}.  In particular, $1\leq L(n)\leq \sqrt{n+1}$, and both bounds are precise \cite{stock}. 
 
The following lemma is a particular case of \cite[Prop.~2]{frid}. 
\begin{lemma}\label{l:l1} 
For all $k_0,\ldots, k_n$ such that $k_i \in \{0,1\}$, the word $f_n^{k_n} f_{n-1}^{k_{n-1}}\cdots f_0^{k_0}$ is a prefix of the Fibonacci word $\bf f$. 
\end{lemma} 
So $L(n)$ is also the number of ways to factor the prefix ${\bf f}(0..n]$ of the Fibonacci word as a sequence of standard words in strictly decreasing order. 
 
To expand this definition, in this note we consider all factorizations of Fibonacci prefixes ${\bf f}(0..n]$ as a concatenation of standard words in (non-strictly) decreasing order. We write $N= [k_n\cdots k_0]_F$ and call this representation of $N$ {\it valid} if $k_i\geq 0$ for all $i$ and ${\bf f}(0..N]=f_n^{k_n} f_{n-1}^{k_{n-1}}\cdots f_0^{k_0}$. Note that according to the previous lemma, every legal representation is valid, but not the other way around. For example, ${\bf f}(0..14]=(abaab)(aba)(aba)(aba)$, making the representation $14= [1300]_F $ valid. Theorem 1 of \cite{frid} says, in particular, that valid representations are exactly those that can be obtained from the canonical one by a series of transformations \eqref{e:k0l}. 
 
Note that a digit of a valid representation cannot exceed 3 since the Fibonacci word does not contain a factor of the form $u^4$ for any non-empty word $u$ \cite{karhumaki}.

The number of valid representations of $N$ is denoted by $V(N)$, and this note is devoted to the study of the sequence $(V(n))$, recently listed in the OEIS as sequence \seqnum{A300066}. Clearly, $V(n)\geq L(n)$, and moreover, we prove an explicit formula for  
$V(n)$ that implies its linear growth. 
 
\section{Result} 
As is well-known, the Fibonacci infinite word 
$${\bf f} = abaababa \cdots $$
is the fixed point of the Fibonacci  
morphism $\mu: a \to ab, b \to a$; moreover, for each $n \geq 1$, we have $f_n=\mu(f_{n-1})$. Consequently, if 
$N= [k_n\cdots k_0]_F $, then Lemma \ref{l:l1} 
implies that  
$$\mu({\bf f}(0..N])=\mu({\bf f}(0.. [k_n\cdots k_0]_F ])=\mu(f_n^{k_n}\cdots f_0^{k_0})=f_{n+1}^{k_n}\cdots f_1^{k_0}={\bf f}(0.. [k_n\cdots k_0 0]_F ].$$ 
 
Let $\varphi$ denote the golden ratio: $\varphi=\frac{1+\sqrt{5}}{2}$. It is important that the Fibonacci word is a Sturmian word of slope $1/(\varphi+1)=1/\varphi^2$ and zero intercept (see Example 2.1.24 of \cite{lothaire2}), that is, for all $n$, we have 
\begin{equation}\label{e:st} 
{\bf f}[n]=\begin{cases} 
        a, & \mbox{~if~} \{n/\varphi^2\}<1-1/\varphi^2;\\ 
        b, & \mbox{~otherwise}. 
       \end{cases} 
\end{equation} 
Here $\{x\} = x- \lfloor x \rfloor $ denotes the fractional part of $x$. 
 
\begin{proposition}\label{p:ab} 
If ${\bf f}[n]=a$, all valid representations of $n$ end with an even number of 0s. If ${\bf f}[n]=b$, all of them end with an odd number of 0s. 
\end{proposition}
\begin{proof} 
It suffices to consult the definition of a valid representation and notice that $f_i$ ends with $a$ if and only if $i$ is even. 
\end{proof} 

We now state our main result.
 
\begin{theorem}\label{t:main} 
If ${\bf f}[n]=a$, then $V(n)=\lceil n/\varphi^2 \rceil$,
or, equivalently,
$V(n)$ is equal to the number of occurrences of $b$ in ${\bf f}(0..n]$, plus one. 
If ${\bf f}[n]=b$, then $V(n)=\lceil n/\varphi^3 \rceil$, or,
equivalently, $V(n)$ is equal to the number of occurrences of $aa$ in
${\bf f}(0..n]$, plus one.
\end{theorem} 
 
To prove the theorem, we will need several more propositions. 
 
\begin{proposition}\label{p:10} 
\leavevmode
\begin{itemize}
\item[(a)]  $V( [r0]_F )\geq V( [r]_F )$ for all $r \in \{ 0,1 \}^*$.
\item[(b)]  For all $k\geq 0$ and all $r' \in \{ 0,1 \}^*$, we have $V( [r'10^{2k+1}]_F )= V( [r'10^{2k}]_F )$. 
\end{itemize}
\end{proposition}
\begin{proof} 
(a):  
Consider a factorization ${\bf f}(0..[r]_F ]=
f_n^{k_n} f_{n-1}^{k_{n-1}}\cdots f_0^{k_0}$. 
Applying the Fibonacci morphism $\mu$ to both sides,
we get the factorization ${\bf f}(0..[r0]_F ] =
f_{n+1}^{k_n} f_{n}^{k_{n-1}}\cdots f_1^{k_0}$.
So the number of factorizations of ${\bf f}(0.. [r0]_F ]$ 
(which is equal to $V( [r0]_F )$) is at least as large as
the number of factorizations of ${\bf f}(0.. [r]_F ]$ (which is
equal to $V( [r]_F )$). 
 
(b) If, in addition $r=r'10^{2k}$ for some $k\geq 0$, we see that 
${\bf f}(0..[r]_F ]$ ends with $f_{2k}$ and ${\bf f}(0..[r0]_F ]$ ends with $f_{2k+1}$,
which in turn ends with $b$. From Proposition \ref{p:ab},
no factorization of ${\bf f}(0..[r0]_F ]$ ends with $f_0$; that is, 
such a factorization must be of the form ${\bf f}(0..[r0]_F ] = f_{n+1}^{k_n} f_{n}^{k_{n-1}}\cdots f_1^{k_0}$.
Taking the $\mu$-preimage,
we get the factorization ${\bf f}(0..[r]_F ] =
f_n^{k_n} f_{n-1}^{k_{n-1}}\cdots f_0^{k_0}$,
thus establishing a bijection and the equality  $V( [r'10^{2k+1}]_F )= V( [r'10^{2k}]_F )$.
\end{proof} 
 
\begin{proposition}\label{p:00} 
We have 
\[V( [z10^{2k}]_F )=V( [ z10^{2k-2}]_F )+V( [z(01)^{k}]_F ).\] 
for all $z\in \{ 0,1 \}^* $ and all $k\geq 1$.
\end{proposition}
\begin{proof}
Proposition \ref{p:ab} tells us that
${\bf f}[ [z10^{2k}]_F ]=a$, and moreover, since $k>0$,
the prefix of length $ [z10^{2k}]_F $ of ${\bf f}$ ends with $aba$,
which is a suffix of $f_{2k}$.
Consider a valid factorization ${\bf f}(0.. [z10^{2k}]_F] = 
f_n^{k_n} f_{n-1}^{k_{n-1}}\cdots f_0^{k_0}$.
If $k_0=0$, then $k_1=0$ since $f_1$ ends with $b$,
so the factorization is of the form
$f_n^{k_n} f_{n-1}^{k_{n-1}}\cdots f_2^{k_2}$.
Taking the $\mu^2$-preimage,
we get a factorization $f_{n-2}^{k_n} f_{n-3}^{k_{n-1}}\cdots f_0^{k_2}$ 
of ${\bf f}(0.. [ z10^{2k-2}]_F] $. Moreover, $\mu^2$ is a bijection between all the factorizations of 
${\bf f}(0.. [ z10^{2k-2}]_F] $ and the factorizations of ${\bf f}(0.. [ z10^{2k}]_F] $ with $k_0=k_1=0$.  
 
On the other hand, if $k_0\neq 0$, then $k_0=1$ since the word that we factor
ends with $aba$. Removing this last occurrence of $f_0=a$, we get the prefix of ${\bf f}$ of length $[ z10^{2k}]_F - 1= [z(01)^k0]_F $. 
From Proposition~\ref{p:10},
the number of valid factorizations of ${\bf f}(0.. [z(01)^k0]_F]$
is equal to that of ${\bf f}(0..  [z(01)^k]_F] $.
Combining the two possibilities, we get the statement of the proposition. 
\end{proof} 
 
\begin{proposition}\label{p:01} 
 For all $z \in \{ 0, 1 \}^* $ and for all $k\geq 1$, we have
\[V( [z10^{k}1]_F )=
\begin{cases}
V( [z10^{k+1}]_F ), & \mbox{~if~} k \mbox{~is odd};\\
V( [z10^k]_F )+V( [z(01)^{k/2}]_F ), &
	\mbox{~if~} k \mbox{~is even.} 
\end{cases}
\] 
\end{proposition}
\begin{proof}
If $k$ is odd, then $[ z10^{k}1]_F = [z10^{k+1}]_F +1$,
and the prefix ${\bf f}(0.. [z10^{k+1}]_F ]$ was considered
in the previous proposition.
It ends with $aba$, and the symbol added to get
${\bf f}(0.. [z10^{k}1]_F ]$ is also $a$.
So ${\bf f}(0.. [z10^{k}1]_F ]$ ends with $abaa$, and all valid factorizations
end with $f_0$.
This means that the number of valid factorizations of 
${\bf f}(0.. [z10^{k}1]_F ]$ is equal to that of ${\bf f}[0..[z10^{k+1}]_F ]$;
that is, $V( [z10^{k}1]_F )=V( [z10^{k+1}]_F )$. 
 
If $k$ is even, $k>0$, then ${\bf f}(0.. [z10^{k}1]_F ]$ ends with
$f_3f_0=abaaba$.
In particular, the last factor of any valid factorization of 
${\bf f}(0.. [z10^{k}1]_F ]$ is either $f_0=a$, or $f_2=aba$.
Indeed, $f_4=abaababa$ and thus
for all $l > 2$ the $f_{2l}$ do not have a common suffix with
${\bf f}(0.. [z10^{k}1]_F ]$. 
So, letting $V_2(n)$ denote the number of factorizations
of ${\bf f }(0..n]$ of the form $f_n^{k_n} f_{n-1}^{k_{n-1}}\cdots f_2^{k_2}$, we get  
\begin{align*} 
 V( [z10^{k}1]_F )&=V( [z10^{k}1]_F -1)+V_2( [z10^{k}1]_F -3)\\ 
&= V( [z10^{k+1}]_F )+V_2( [z(01)^{k/2}00]_F )\\ 
&= V( [z10^{k}]_F )+V( [z(01)^{k/2}]_F ). 
\end{align*} 
Here the last equality follows from Proposition \ref{p:10} (for the first addend) and by taking $\mu^{-2}$ of each factorization (for the second one). 
\end{proof} 
 
Propositions \ref{p:10} to \ref{p:01} give a full list of recurrence relations sufficient to compute $V(n)$ for every $n>1$, starting from $V(1)=1$. Before using them to prove the main theorem, we consider two particular cases. 
 
\begin{corollary} 
For all $k \geq 1$ we have
\[V(F_{2k+1}-1)=V(F_{2k+1}-2)=F_{2k-1}\] 
and 
\[V(F_{2k+2}-2)=F_{2k}\] 
\end{corollary} 
\begin{proof}
For $k=1$, the equalities can be easily checked:
$V(F_3-1)=V(1)=V(F_3-2)=V(0)=1=F_1$, and $V(F_4-2)=V(1)=1=F_2$.
We also observe that $F_{2k+1}-1= [(10)^{k-1} 1]_F$,
$F_{2k+1}-2= [(10)^{k-1} 0]_F $, and
$F_{2k+2}-2= [(10)^{k-1} 01]_F $.
Now we assume that the equalities hold for $k$, and
use Propositions \ref{p:00} and \ref{p:01} to prove they hold for $k+1$: 
\begin{align*} 
 V(F_{2k+3}-2)&=V( [(10)^k0]_F )=V( [ (10)^{k-1} 1]_F )+V( [(10)^{k-1} 01]_F )\\ 
&=V(F_{2k+1}-1)+V(F_{2k+2}-2)=F_{2k-1}+F_{2k}=F_{2k+1},\\ 
V(F_{2k+3}-1)&=V( [ (10)^{k}1]_F )=
V( [(10)^{k}0]_F )=V(F_{2k+3}-2)=F_{2k+1},\\ 
V(F_{2k+4}-2)&=V( [ (10)^{k}01]_F )=
V( [ (10)^{k}0]_F ) + 
V( [(10)^{k-1} 01]_F  )\\ 
&=V(F_{2k+3}-2)+V(F_{2k+2}-2)=F_{2k+1}+F_{2k}=F_{2k+2}. 
\end{align*}
\end{proof}

\begin{corollary} 
For all $k\geq 1$, we have
\[V(F_{2k})=V(F_{2k+1})=F_{2k-2}+1.\] 
\end{corollary} 
\begin{proof}
For $k=1$, the equalities can be easily checked: $V(F_2)=V(1)=V(F_3)=V(2)=1=F_0+1$. Suppose the equalities hold for $k$; let us prove them for $k+1$. With Proposition \ref{p:00}, we have 
\begin{displaymath}
 V(F_{2k+2})=
 V( [ 10^{2k}]_F )=
 V([10^{2k-2}]_F )+V( [(10)^{k-1} 1]_F )=F_{2k-2}+1+F_{2k-1}=F_{2k}+1, 
\end{displaymath}
and with Proposition \ref{p:10}, we have 
\[V(F_{2k+3})=V( [10^{2k+1}]_F )=V( [10^{2k}]_F )=V(F_{2k+2})=F_{2k}+1.  \]
\end{proof}

\begin{proposition}\label{p:23} 
Let $n= [z]_F $ and $n'= [z0]_F $ be such that ${\bf f}[n]=a$.
Then $\lceil n/\varphi^2 \rceil = \lceil n'/\varphi^3 \rceil$.  
\end{proposition} 
\begin{proof}
Let us write the canonical Fibonacci representation of
$n$ as $\sum_{1\leq i \leq l} F_{m_i}$, where $2\leq m_1<m_2<\cdots<m_l$.
Since ${\bf f}[n]=a$, from Proposition \ref{p:ab} we get that $m_1$ is even.  
 
Now $F_k=\frac{1}{\sqrt{5}}(\varphi^{k}-\psi^{k})$,
where $\psi=\frac{1-\sqrt{5}}{2}$, $-1<\psi<0$.
So 
$$n=\sum_{1\leq i \leq l} F_{m_i}=\frac{1}{\sqrt{5}} \left(\sum_{1\leq i \leq l} \varphi^{m_i} - \sum_{1\leq i \leq l} \psi^{m_i} \right)$$
and
$$n'=\sum_{1\leq i \leq l} F_{m_i+1}=\frac{1}{\sqrt{5}} \left(\sum_{1\leq i \leq l} \varphi^{m_i+1} - \sum_{1\leq i \leq l} \psi^{m_i+1}\right),$$
implying that  
\[\frac{n'}{\varphi}= \frac{1}{\sqrt{5}} \left (\sum_{1\leq i \leq l} \varphi^{m_i} - \frac{1}{\varphi}\sum_{1\leq i \leq l} \psi^{m_i+1}\right ).\] 
The difference between the two values is 
\[\frac{n'}{\varphi}-n=\frac{1}{\sqrt{5}}\left( 1-\frac{\psi}{\varphi}\right ) S,\] 
where 
$$S=\sum_{1\leq i \leq l} \psi^{m_i}=\psi^{m_1}\sum_{1\leq i \leq l} \psi^{m_i-m_1}.$$
 
Let us estimate $S$. Since $m_1\geq 2$, $m_1$ is even and $0<\psi^{m_1}<\psi^2$, an 
upper bound for $S$ is 
\[S<\psi^{m_1}\sum_{k=0}^{\infty} \psi^{2k}=\frac{\psi^{m_1}}{1-\psi^2}\leq \frac{\psi^2}{1-\psi^2},\] 
whereas a lower bound is  
\[S>\psi^{m_1}\left(1+\sum_{k=1}^{\infty} \psi^{2k+1}\right )>\psi^{m_1}\left(1+\sum_{k=0}^{\infty} \psi^{2k+1}\right )=\psi^{m_1}\left( 1+\frac{\psi}{1-\psi^2} \right) = 0.\] 
So
\[0<\frac{n'}{\varphi}-n<\frac{\psi^2}{\sqrt{5}}\left( 1-\frac{\psi}{\varphi}\right )\frac{1}{1-\psi^2}=\frac{1}{\varphi^2}.\] 
Dividing by $\varphi^2$, we get
\[0<\frac{n'}{\varphi^3}-\frac{n}{\varphi^2}<\frac{1}{\varphi^4}<\frac{1}{\varphi^2}.\]
Together with \eqref{e:st}, meaning that $\{n/\varphi^2\}<1-1/\varphi^2$, the last inequality implies the statement of the Proposition. 
\end{proof} 
 
\begin{proof}[Proof of Theorem \ref{t:main}.] Let us start with the case of $
{\bf f}[n]=a$ and proceed by induction starting with $V(1)=1$. For $n>1$, there are three subcases:
\begin{itemize}
\item[(a)] $n= [z10^{2k}]_F $, $k>0$;
\item[(b)] $n= [z10^{k}1]_F $, $k$ odd;
\item[(c)] $n= [ z10^{k}1]_F $, $k$ even. 
\end{itemize}
From now on we suppose that the statement of the theorem holds for all $n',n''<n$. 

\noindent (a) Since $n= [z10^{2k}]_F$ and $k>0$, Proposition \ref{p:00} gives 
$V(n)=V( [z10^{2k}]_F ) = V( [ z10^{2k-2}]_F )+
V( [ z(01)^{k}]_F )$. 
Write $ [z10^{2k-2}]_F =n'$ and $ [z(01)^{k}]_F =n''$.
Note that Proposition \ref{p:ab} gives ${\bf f}[n']={\bf f}[n'']=a$.
At the same time, $n''+1= [z10^{2k-1}]_F $ and thus ${\bf f}[n''+1]=b$.
Now \eqref{e:st} implies that $\{n'/\varphi^2\} \in (0,1-1/\varphi^2)$
and $\{(n''+1)/\varphi^2\}\in (1-1/\varphi^2,1)$.
Also, the Fibonacci representation of $n'$ is obtained from that of
$n''+1$ by a one-symbol shift to the left. So, summing up $n'$ and $n''+1$, 
due to the Fibonacci recurrence relation, we get the number with the same representation but shifted to the left yet another position, meaning that $n'+n''+1=n$.  
 
Let us consider the sum $t = \{n'/\varphi^2\}+\{(n''+1)/\varphi^2\}$.
From the inclusions above, we see that $t$ belongs to the interval $(1-1/\varphi^2,2-1/\varphi^2)$. But we also know that $\{n/\varphi^2\}=\{(n'+n''+1)/\varphi^2\}\in (0,1-1/\varphi^2)$, since ${\bf f}[n]=a$. So  
\[\{n/\varphi^2\}=\{n'/\varphi^2\}+\{(n''+1)/\varphi^2\}-1,\]  
which is equivalent to 
$\lfloor n/\varphi^2 \rfloor=\lfloor n'/\varphi^2 \rfloor+\lfloor (n''+1)/\varphi^2\rfloor +1$ and to 
$\lfloor n/\varphi^2 \rfloor=\lfloor n'/\varphi^2 \rfloor+\lfloor n''/\varphi^2\rfloor +1$ (since $\lfloor n''/\varphi^2\rfloor=\lfloor (n''+1)/\varphi^2\rfloor$). Since all the numbers under consideration are irrational,
and thus every ceiling is just the floor plus 1, we get
\[\lceil n/\varphi^2 \rceil=\lceil n'/\varphi^2 \rceil+\lceil n''/\varphi^2\rceil.\] 
To establish the statement of the theorem for this subcase, it is sufficient to use Proposition \ref{p:00} and the induction hypothesis: $V(n')=\lceil n'/\varphi^2 \rceil$ and $V(n'')=\lceil n''/\varphi^2 \rceil$.  

\bigskip
 
\noindent (b):  Here $n= [z10^{2k-1}1]_F $ and $k>0$.
It suffices to refer to the previous subcase and to 
Proposition \ref{p:01}:
$V(n)=V(n-1)=V( [z10^{2k}]_F )=\lceil (n-1)/\varphi^2 \rceil$.
It remains to notice that $\lceil (n-1)/\varphi^2 \rceil=\lceil n/\varphi^2 \rceil$, since ${\bf f}[n-1]=a$. 

\bigskip
 
\noindent (c):  Here $n= [z10^{2k}1]_F $ and $k>0$.
We use Proposition \ref{p:01}: $V( [z10^{2k}1]_F )=V( [z10^{2k}]_F )+
V( [z(01)^{k}]_F )$. As above, write $n'= [z10^{2k}]_F $ and
$n''= [z(01)^{k}]_F $; then $n=n'+n''+2$, whereas $V(n)=V(n')+V(n'')$. By the induction hypothesis, $V(n')=\lceil n'/\varphi^2 \rceil$ and $V(n'')=\lceil n''/\varphi^2 \rceil$. 
 
We have ${\bf f}[n]=a$ and ${\bf f}[n-1]=b$, implying from \eqref{e:st} that $\{(n-1)/\varphi^2\}\in (1-1/\varphi^2,1)$ and thus $\{n/\varphi^2\}\in (0,1/\varphi^2)$. At the same time, ${\bf f}[n']={\bf f}[n'']=a$ implies $\{n'/\varphi^2\},\{n''/\varphi^2\}\in (0,1-1/\varphi^2)$ and thus  
\[\{n'/\varphi^2\}+\{n''/\varphi^2\}+\{2/\varphi^2\}\in (2/\varphi^2,2).\] 
Comparing it to $\{n/\varphi^2\}=\{(n'+n''+2)/\varphi^2\}\in (0,1/\varphi^2)$, we see that 
\[\{n/\varphi^2\}=\{n'/\varphi^2\}+\{n''/\varphi^2\}+\{2/\varphi^2\}-1.\] 
But since $n=n'+n''+2$ and $x=\lfloor x \rfloor + \{x\}$ for every $x$, this also means that 
\[\lfloor n/\varphi^2\rfloor=\lfloor n'/\varphi^2\rfloor +\lfloor n''/\varphi^2\rfloor +1.\] 
Finally, since $k/\varphi^2$ is not an integer for any integer $k>0$,
we have
$\lceil k/\varphi^2 \rceil = \lfloor k/\varphi^2\rfloor +1$, so that  
\[\lceil n/\varphi^2 \rceil=\lceil n'/\varphi^2 \rceil+\lceil n''/\varphi^2\rceil.\] 
It remains to use the induction hypothesis to establish 
\[V(n)=V(n')+V(n'')=\lceil n'/\varphi^2 \rceil+\lceil n''/\varphi^2\rceil=\lceil n/\varphi^2 \rceil,\] 
which was to be proved. 
 
To complete the part of the proof concerning ${\bf f}[n]=a$, it remains to notice that $\lceil n/\varphi^2 \rceil$ is equal to the number of $b$s in ${\bf f}(0..n]$ plus one, due to \eqref{e:st}. 
 
Now for ${\bf f}[n]=b$, it is sufficient to combine Propositions \ref{p:ab}, \ref{p:10} and \ref{p:23}: if ${\bf f}[n]=b$, then $n=[r0]_F $, where $m = [r]_F $ and ${\bf f}[m]=a$. Then  
\[V(n)=V(m)=\lceil m/\varphi^2 \rceil = \lceil n/\varphi^3 \rceil.\]  
Here ${\bf f}(0..n]=\mu({\bf f}(0..m])$, and so the  occurrences of $aa$ in $
{\bf f}(0..n]$ correspond exactly to occurrences of $b$ in $\mu({\bf f}(0..m])$. 
The theorem is proved. 
\end{proof} 
 
The theorem ensures
that the sequence $(V(n))$ grows as depicted in Fig.~\ref{f:1}. The two visible straight lines correspond to the symbols of the Fibonacci word equal to $a$ (the upper line) or $b$ (the lower line). 

\section{Fibonacci-regular representation}

A sequence $(s(n))_{n \geq 0}$ is said to be {\it Fibonacci-regular}
if there exist an integer $k$, a
row vector $v$ of dimension $k$,
a column vector $w$ of dimension $k$,
and a $k \times k$ matrix-valued morphism $\rho$ on $\{0, 1\}^*$
such that 
$$ s( [z]_F ) = v \rho(z) w $$
for all canonical Fibonacci representations $z \in L_V$.
The triple $(v, \rho, w)$ is called a {\it linear representation};
see, for example, \cite{mss}.

Berstel \cite{berstel_fib} gave the following
linear representation for the function $L(n)$ we mentioned
previously in Section~\ref{notation}:
$$
v = [1 \ 0 \ 0 \ 0 ],
	 \quad
\rho(0) = \left[ \begin{array}{cccc}
	1 & 0 & 0 & 0 \\
	0 & 0 & 1 & 0 \\
	1 & 1 & 0 & 0 \\
	1 & 0 & 0 & 0
	\end{array} \right], \quad
\rho(1) = \left[ \begin{array}{cccc}
	0 & 1 & 0 & 1 \\
	0 & 0 & 0 & 0 \\
	0 & 1 & 0 & 0 \\
	0 & 0 & 0 & 0
	\end{array} \right], \quad
w = \left[ \begin{array}{c}
	1 \\
	0 \\
	0 \\
	1
	\end{array}
	\right].
$$
Hence $L(n)$ is Fibonacci-regular.

We can find a similar representation for the function $V(n)$.
For technical reasons it is easier to deal with the reversed
Fibonacci representation; one can then obtain the ordinary
linear representation by interchanging the roles of the vectors
and taking the transposes of the matrices.

\begin{theorem}
$V(n)$ has the reversed linear representation $(t, \gamma, u)$, where
\begin{align*}
t &= [ 1\ 0 \ 0 \ 0 \ 0 \ 0 \ 0 \ 0 ], \quad &
u &= [ 1\ 1 \ 1 \ 1 \ 1 \ 2 \ 1 \ 4 ]^T \\
\gamma(0) & = \left[ \begin{array}{rrrrrrrr}
	0& 1& 0& 0& 0& 0& 0& 0 \\
	0& 0& 0& 1& 0& 0& 0& 0 \\
	-1& 1& 0& 1& 0& 0& 0& 0\\
	0& 0& 0& 0& 1& 0& 0& 0\\
	0& 0& 0& 0& 0& 0& 1& 0 \\
	-1& 0& 0& 2& 1& 0& 0& 0\\
	 1&-1& 0&-3& 3& 0& 1& 0\\
	-1&-1& 0& 2& 3& 0& 1& 0\\
	\end{array} \right],  &
\gamma(1) & = \left[ \begin{array}{cccccccc}
0&0&1&0&0&0&0&0\\
0&0&1&0&0&0&0&0\\
0&0&0&0&0&0&0&0\\
0&0&0&0&0&1&0&0\\
0&0&0&0&0&1&0&0\\
0&0&0&0&0&0&0&0\\
0&0&0&0&0&0&0&1\\
0&0&0&0&0&0&0&0\\
	\end{array} \right] .
\end{align*}
\end{theorem}

\begin{proof}
Define $g(x) = V([x]_F)$ if $x$ is a valid canonical representation
(that is, containing no leading zeroes, and no two consecutive 1's),
and $0$ otherwise.
It suffices to show, for all $x \in \{ 0,1 \}^*$ and $i \in \{ 0,1 \}$, that
\begin{equation}
\left[
\begin{array}{c}
g(xi) \\
g(xi0) \\
g(xi1) \\
g(xi00) \\
g(xi000) \\
g(xi100) \\
g(xi0000) \\
g(xi10000)  
\end{array} \right]
= \gamma(i) 
\left[ \begin{array}{c}
g(x) \\
g(x0) \\
g(x1) \\
g(x00) \\
g(x000) \\
g(x100) \\
g(x0000) \\
g(x10000) 
\end{array} \right] .
\label{mat1}
\end{equation}
Once we prove this, it is then easy to see (using induction on $|z|$) 
that, if $z$ is the Fibonacci representation of $n$, then
$t \gamma( z^R) u = V( n )$, where $z^R$ is the reversal of $z$.

Thus it suffices to verify Eq.~\eqref{mat1}.  This is equivalent to proving
the following identities for $x$. 
\begin{align}
g(x01) &= -g(x) + g(x0) + g(x00) \label{e1} \\
g(x10) &= g(x1) \label{e5} \\
g(x0100) &= -g(x) + 2g(x00) + g(x000) \label{e2} \\
g(x1000) &= g(x100) \label{e6}\\
g(x010000) &= -g(x)- g(x0) + 2g(x00) + 3g(x000) + g(x0000) \label{e4} \\
g(x00000) &= g(x) - g(x0) -3g(x00) + 3g(x000) + g(x0000) \label{e3}.
\end{align}

Identities \eqref{e5} and \eqref{e6} are particular cases of Proposition \ref{p:10} (b). 

To prove \eqref{e1}, consider separately two cases: if $x$ ends with an even number of zeros, then $g(x)=g(x0)$ due to Proposition \ref{p:10} (b) and $g(x00)=g(x01)$ due to Proposition \ref{p:01}, so the identity holds. If $x$ ends with an odd number of zeros, $x=z10^{2k+1}$, $k \geq 0$, then due to Proposition \ref{p:01},
\[g(x01)=g(z10^{2k+2}1)=g(z10^{2k+2})+g(z(01)^{k+1})=g(x0)+g(z(01)^{k+1}).\]
On the other hand, due to Propositions \ref{p:10} and \ref{p:00},
\[g(x00)=g(x0)=g(z10^{2k+2})=g(z10^{2k})+g(z(01)^{k+1})=g(x)+g(z(01)^{k+1}).\]
Comparing these equalities, we get \eqref{e1}.

To prove \eqref{e2}, it is sufficient to use Proposition \ref{p:00} to get
\[g(x0100)=g(x01)+g(x001),\]
and then to use \eqref{e1} twice, for $g(x01)$ and for $g(x001)$.

To prove \eqref{e4}, it is sufficient to use Propositions \ref{p:00} and \ref{p:10} to get
\[g(x010000)=g(x0100)+g(x00101)=g(x0100)+g(x00100).\]
Now \eqref{e4} is obtained immediately by summing up \eqref{e2} applied to $x$ and to $x0$.

Finally, to prove \eqref{e3}, we again have to consider two cases. If $x=z10^{2k}$, $k \geq 0$, then due to Proposition \ref{p:10}, $g(x0^5)=g(x0000)$, $g(x000)=g(x00)$, $g(x0)=g(x)$, and the equality holds. If now $x=z10^{2k+1}$, $k \geq 0$, then \eqref{e3} immediately simplifies with Proposition \ref{p:10} as
\[g(z10^{2k+6})-g(z10^{2k+4})= 3[g(z10^{2k+4})-g(z10^{2k+2})]-[g(z10^{2k+2})-g(z10^{2k})].\]
Applying Proposition \ref{p:00}, we reduce it to
\[g(z(01)^{k+3})=3g(z(01)^{k+2})-g(z(01)^{k+1}),\]
or, writing $y = z(01)^{k+1}$ and applying Proposition \ref{p:01} again,
\[g(y0100)=3g(y00)-g(y).\]
But this is exactly \eqref{e2} since $g(y00)=g(y000)$.

\end{proof}

\end{document}